\documentclass[a4paper,10pt]{article}
\usepackage[utf8x]{inputenc}

\usepackage{amsfonts,amsthm,amsmath,amssymb}

\usepackage{graphicx}
\newcommand{\R}{\mathbb{R}}

\newcommand{\Z}{\mathbb{Z}}

\newcommand{\M}{\mathrm{Mod}}
\newcommand{\Ho}{\mathrm{Homeo}}
\newcommand{\Diff}{\mathrm{Diff}}

\newcommand{\Aut}{\mathrm{Aut}}
\newcommand{\Out}{\mathrm{Out}}
\newcommand{\Inn}{\mathrm{Inn}}
\newcommand{\id}{\mathrm{id}}
\newcommand{\Gal}{\mathrm{Gal}}
\newcommand{\Q}{\mathbb{Q}}

\newcommand{\C}{\mathbb{C}}
\newcommand{\GL}{\mathrm{GL}}
\newcommand{\SL}{\mathrm{SL}}

\newcommand{\ab}{\mathrm{ab}}

\newtheorem{prop}{Proposition}
\newtheorem{lemma}{Lemma}

\newtheorem{rem}{Remark}
\title{Representations of the mapping class group of a non-orientable surface}
\author{Ferit Deniz and Wilhelm Singhof}
\date{}
\begin{document}

\maketitle
\begin{abstract}
By considering appropriate finite covering spaces of closed non-orientable surfaces, we construct linear representations of their mapping class group which have finite index image in certain big arithmetic groups.

\end{abstract}

\section{Introduction}\label{intro}

In this paper, we study representations of the mapping class group $\M(N_{g+1})$ of the closed non-orientable surface $N_{g+1}$.
Corresponding (and better) results for orientable surfaces were obtained by Looijenga \cite{Lo} and for automorphism groups of
free groups by Grunewald-Lubotzky \cite{GL}. Our main result is the following:

\medskip
\textbf{Theorem.}
\emph{ Let $g=2r \geq 4$ and let $k$ be a natural number. Let $R$ be the ring of integers in the number field $\Q(e^{2\pi i/k})$. There is
 a subgroup $\Gamma$ of $\M(N_{g+1})$ of finite index and a homomorphism 
 \[\sigma: \Gamma \rightarrow \SL(g-1,R)^k\]
 such that $\sigma(\Gamma)$ is of finite index in $\SL(g-1,R)^k$.}

\medskip
Our approach is basically the same as in \cite{Lo} and \cite{GL}:
We look at certain finite coverings $\tilde{N}$ of $N:=N_{g+1}$. Let $G$ be the group of deck transformations. The subgroup $\Gamma$
consists roughly of those elements of $\M(N)$ which lift to $\tilde{N}$ in a $G$-equivariant manner, and the representation $\sigma$
is obtained from the canonical action of $\Gamma$ on the homology group $H_1(\tilde{N})$.

\medskip
Let us point out the main differences to \cite{Lo} and \cite{GL}: Looijenga took any finite abelian covering $\tilde{\varSigma}$ of an
orientable surface $\varSigma$ and obtained an arithmeticity result for the representation on $H_1(\tilde{\varSigma})$. For rather
trivial reasons, such a result is not true in the non-orientable case; as we shall see, the orientation covering is a counter-example.
To prove our theorem, we look at special coverings for which the group $G$ is a product of two cyclic groups. 

\medskip
On the other hand, the work of Grunewald and Lubotzky on the automorphism groups of free groups is more general in that they allow 
arbitrary, not necessarily abelian, finite groups $G$. In fact, one of their main results is obtained by taking for $G$ a product
of two cyclic groups and a symmetric group. We shall study such a situation for mapping class groups in a subsequent 
paper in which we also treat the case of odd genus.

\medskip
Here is a short outline of the contents of the various sections of the paper:

\medskip
In section 2, we recall first how a covering $\tilde{M}$ of a surface $M$, given by an epimorphism $v:\pi \rightarrow G$ from the
fundamental group $\pi$ of $M$ onto a finite group $G$, leads to a representation of a subgroup $\Gamma(v)$ of finite index of
$\Aut\,\pi$ on the $G$-space $H_1(\tilde{M})$; we give two different interpretations of $\Gamma(v)$. Then we analyze the structure
of the rational homology of $\tilde{M}$ as a $\Q[G]$-module. In the situation which is relevant for the rest of the paper in which
$M=N_{g+1}$ and $\tilde{M}$ is non-orientable, it is easy to see that 
\[H_1(\tilde{M};\Q) \cong \Q \oplus \Q[G]^{g-1}\;.\]

This explains why the number $g-1$ appears in our theorem. From that point on, we consider only the case $g=2r$ and take a particularly
simple homomorphism $v$ from $\pi$ to a product of two cyclic groups.

\medskip
In section 3, we take a closer look at the orientation covering $\varSigma_g \rightarrow N_{g+1}$. We can consider $\M(N_{g+1})$ as a
subgroup of $\M(\varSigma_g)$; we work out how certain elements of $\M(N_{g+1})$ which are known to be generators of $\M(N_{g+1})$
operate on the fundamental group $\pi$. We use this in section 4 to complete the proof of the main theorem by producing sufficiently
many elements in the subgroup $\Gamma(v)$ of $\pi$.

\medskip
We are grateful to Alex Lubotzky for sending us a preliminary version of work on representations of mapping class groups of orientable surfaces. 

\section{The construction of representations}\label{constr}

\subsection{The general setting}\label{general}

Let $\pi$ be any group, let $G$ be a finite group and $v:\pi \rightarrow G$ a surjective homomorphism with kernel $U$. We denote by
$q:U \rightarrow U^{\ab}$ the canonical projection and obtain a linear operation $G \times U^{\ab} \rightarrow U^{\ab}$ by
\begin{eqnarray}
v(a) \cdot q(u) := q(aua^{-1})  \label{S0} 
\end{eqnarray}

for $a \in \pi$ and $u \in U$. Let
\begin{eqnarray}
 \Gamma (v):= \{\varphi \in \Aut\, \pi \mid v \circ \varphi = v\}\; . \label{S1}
\end{eqnarray}

If the group $\pi$ is finitely generated, the subgroup $\Gamma (v)$ has finite index in $\Aut\, \pi$.

\smallskip
For $\varphi \in \Gamma (v)$, we have $\varphi (U)=U$, and $\varphi$ defines an automorphism $\bar{\varphi}$ of $U^{\ab}$ which is
$G$-equivariant. By $\varphi \mapsto \bar{\varphi}\, \otimes\, \id$, we get a homomorphism 
\begin{eqnarray}
 \rho_v: \Gamma(v) \rightarrow \Aut_{\C} (U^{\ab} \otimes_{\Z} \C) \; . \label{S2}
\end{eqnarray}

Let $H \subseteq U^{\ab} \otimes_{\Z} \C$ be an isotypical component of $U^{\ab} \otimes_{\Z} \C$ considered as a representation of $G$.
We have $\bar{\varphi}\, \otimes \,\id(H) \subseteq H$,and thus we obtain a representation 
\begin{eqnarray}
 \rho_v^H: \Gamma(v) \rightarrow \Aut_{\C} \, H\;. \label{S3}
\end{eqnarray}

Instead of $\C$, we can of course take an appropriate number field. This, so far, is a direct generalization of the basic construction of
\cite{GL} from free groups to arbitrary groups.

\medskip
Now suppose that $\pi =\pi_1(M,y_0)$ is the fundamental group of a good topological space, say a connected CW complex. Let $p: \tilde{M}
\rightarrow M$ be a covering space of $M$, with $\tilde{M}$ connected, such that 
\[p_{\ast} (\pi_1(\tilde{M}, x_0))=U\]
for each $x_0 \in p^{-1}(y_0)$. The homomorphism $v$ provides us with an isomorphism between $G$ and the group of deck transformations
of $p$. Under this isomorphism, the operation (\ref{S0}) of $G$ on $U^{\ab}$ becomes the usual operation of deck transformations on 
$H_1(\tilde{M};\Z)$.

\smallskip
Let $\Ho (M,y_0)$ be the group of all homeomorphisms $f:M \rightarrow M$ with $f(y_0)=y_0$. For $f \in \Ho (M,y_0)$, let $f_{\ast} \in 
\Aut\,\pi$ be the induced automorphism. We fix a point $x_0 \in p^{-1}(y_0)$. Let us say that an element $\tilde{f} \in \Ho(\tilde{M},x_0)$
\emph{lies above} $f$ if $p \circ \tilde{f}= f \circ p$. By standard covering space theory, we have:

\begin{rem}\label{Rem}
 For $f \in \Ho(M,y_0)$, the following two assertions are equivalent:
 \begin{itemize}
  \item[(i)] $f_{\ast} \in \Gamma(v)$.
  \item[(ii)] There exists $\tilde{f} \in \Ho(\tilde{M},x_0)$ which lies above $f$ and is $G$-equivariant.
 \end{itemize}
\end{rem}

Let now $M$ be closed surface, orientable or not. We begin by showing that, as for free groups in \cite{GL}, there is a homological
criterion for an automorphism of $\pi$ to belong to $\Gamma(v)$.

\begin{lemma}\label{equiv}
Let $M$ be a closed surface of genus $\geq 2$. For $f \in \Diff(M,y_0)$, the following assertions are equivalent:
 
 \begin{itemize}
  \item[(i)] $f_{\ast} \in \Gamma(v)$.
  \item[(ii)] There exists $\tilde{f} \in \Diff(\tilde{M},x_0)$ lying above $f$ such that the induced automorphism $\tilde{f}_{\ast}$ of $H_1(\tilde{M};\C)$
  is $G$-equivariant.
 \end{itemize}

\end{lemma}

\emph{Proof}. We have to show that \emph{(ii)} implies \emph{(i)}. By (\ref{Rem}), this means that we have to show that, for an element
 $g \in G$ considered as a diffeomorphism of $\tilde{M}$, we have $\tilde{f} \circ g = g \circ \tilde{f}$. Since $\tilde{f} \circ g$ lies
 above $f$, there exists $h \in G$ with $\tilde{f} \circ g = h \circ \tilde{f}$, hence 
\[\tilde{f}_{\ast} \circ g_{\ast} = h_{\ast} \circ \tilde{f}_{\ast}\; .\]

On the other hand, equivariance of $\tilde{f}_{\ast}$ means 
\[\tilde{f}_{\ast} \circ g_{\ast} = g_{\ast} \circ \tilde{f}_{\ast}\; .\]

We conclude that $h_{\ast} = g_{\ast}$. Let $\gamma:= hg^{-1}$. Then $\gamma_{\ast} = \id$, and we have to show that $\gamma = \id$.

\smallskip
We choose a triangulation of $\tilde{M}$ which is $G$-invariant and for which the fixed point set $F$ of $\gamma$ is a subcomplex. Assume
that $\gamma \neq \id$. Then $F$ is a finite disjoint union of points and circles, hence $\chi(F) \geq 0$. Consider the fixed point index

\[I(\gamma) = \sum_i (-1)^i \; \text{trace} \, (\gamma_{\ast} \mid H_i(\tilde{M};\C))\; .\]

Computing $I(\gamma)$ by means of the simplicial chain complex of $\tilde{M}$, we see that 
\[I(\gamma) = \chi(F) \geq 0\;.\]

On the other hand, since $\gamma_{\ast} \mid H_1(\tilde{M};\C)$ is the identity and since the genus of $M$ and hence of $\tilde{M}$ is
$\geq 2$, we have $I(\gamma) < 0$. This shows that $\gamma = \id$. \hfill $\Box$

\begin{prop}\label{gamma}
 Let $\pi$ be the fundamental group of a closed surface of genus $\geq 2$, let $G$ be a finite group and $v:\pi \rightarrow G$ an
 epimorphism with kernel $U$. For $\varphi \in \Aut\,\pi$, the following assertions are equivalent:
 \begin{itemize}
  \item[(i)] $\varphi \in \Gamma(v)$.
  \item[(ii)] $\varphi (U) = U$ and $\varphi$ induces a $G$-equivariant automorphism $\bar{\varphi}\,\otimes\,\id$ of\linebreak $U^{\ab}\otimes_{\Z}\C$.
 \end{itemize}

\end{prop}

\begin{proof}
 We have to show that (ii) implies (i). By the Dehn-Nielsen theorem, there exists $f \in \Diff(M,y_0)$ with $\varphi =f_{\ast}$. Since $\varphi (U)=U$, there
 exists $\tilde{f} \in \Diff(\tilde{M},x_0)$ lying above $f$. By assumption, $\tilde{f}_{\ast} \in \Aut\,H_1(\tilde{M};\C)$ is equivariant.
 Hence $\varphi \in \Gamma(v)$ by Lemma \ref{equiv}.
\end{proof}

\begin{prop}\label{structure}
Let $M$ be a closed surface with fundamental group $\pi$, let $G$ be a finite group and let $p:\tilde{M}\rightarrow M$ be the covering of $M$
belonging to the epimorphism $v:\pi \rightarrow G$. As a $\Q[G]$-module, the rational homology of $\tilde{M}$ has the following structure:
\begin{itemize}
 \item[(a)] $H_1(\tilde{M})\cong \Q^2 \oplus \Q[G]^{2g-2}$ if $M=\varSigma_g$.
 \item[(b)] $H_1(\tilde{M})\cong \Q \oplus \Q[G]^{g-1}$ if $M=N_{g+1}$ and if $\tilde{M}$ is non-orientable.
 \item[(c)] $H_1(\tilde{M})\cong \Q\oplus \Q^- \oplus \Q[G]^{g-1}$ if $M=N_{g+1}$ and if $\tilde{M}$ is orientable.
\end{itemize}
Here, $\Q^-$ is the one-dimensional representation of $G$ defined as follows: Let $H$ be the (unique) subgroup of index $2$ of $\pi$
which is isomorphic to $\pi_1(\varSigma_g)$. Then $v(H)$ is a subgroup of index $2$ of $G$. Let $\Q^-$ be the representation with kernel
$v(H)$.
\end{prop}

\begin{proof}
We show (b) which is the assertion which is relevant for the purposes of the present paper; (a) and (c) are similar. Choose an open disk $D$
in $M$, let $X:=M \smallsetminus D$ and $\tilde{X}:=p^{-1}(X)$. Look at the exact sequence of the pair $(\tilde{M},\tilde{X})$:
\[0 \rightarrow H_2(\tilde{M},\tilde{X}) \rightarrow H_1(\tilde{X})\rightarrow H_1(\tilde{M})\rightarrow H_1(\tilde{M},\tilde{X})\rightarrow 0\;.\]
This is an exact sequence of $\Q[G]$-modules. The fundamental group of $X$ is a free group of rank $g+1$, and $\tilde{X}$ is connected. Hence,
by Gaschütz theory \cite{Ga}, as in \cite{GL}, we have 
\[H_1(\tilde{X}) \cong \Q\oplus \Q[G]^g\,.\]
We have $H_i(\tilde{M},\tilde{X}) \cong H_i(\tilde{M}/\tilde{X})$, and $\tilde{M}/\tilde{X}$ is a wedge of $|G|$ copies of $S^2$ with 
$H_2(\tilde{M}/\tilde{X}) \cong \Q[G]$ and $H_1(\tilde{M}/\tilde{X})=0$.
\end{proof}

From now on, we assume that $M$ is a non-orientable surface. To prove our main theorem, we consider the case in which $G$ is a product of two
cyclic groups. To do that, we have first to look at the orientation covering of $M$. This is of course also a special instance of our 
general setting; however, it doesn't lead to new representations.

\subsection{The orientation covering}\label{orient}

We consider the orientation covering $p:\varSigma \rightarrow N$ of the closed non-orientable surface $N:= N_{g+1}$. The total space
$\varSigma$ is the orientable surface of genus $g$. The image of $p_{\ast}: \pi_1(\varSigma,x_0) \rightarrow \pi_1(N,y_0)=:\pi$ is the
unique characteristic subgroup $U$ of $\pi$ of index $2$.

\smallskip
Let us briefly illustrate the general method of Section \ref{general} in this situation: The corresponding group $G$ is the cyclic group
$C_2$ and $v:\pi \rightarrow C_2$ is the homomorphism with kernel $U$. We have $\Gamma(v)= \Aut \,\pi$; this is quite exceptional. Let
$j:\varSigma \rightarrow \varSigma$ be the non-trivial deck transformation. Then 
\[H_1 (\varSigma; \Q)= H^+ \oplus H^-\, ,\]
where $H^{\pm}$ is the $\pm 1$-eigenspace of $j_{\ast}$. It is easy to see that
\[\dim H^+ = \dim H^- =g\]
and that $p_{\ast} \mid H^+$ is an isomorphism of $H^+$ onto $H_1(N; \Q)$. By (\ref{S3}) we get then two $g$-dimensional representations 
\[
\rho^{\pm}:=\rho_v^{H^{\pm}}: \Aut \, \pi \rightarrow \Aut_{\Q}\,H^{\pm}\; .
\]
The representation $\rho^+$ is the standard representation of $\Aut \, \pi$ on $H_1(N;\Q)$, and it is easy to see that $\rho^-$ is 
equivalent to the dual representation of $\rho^+$. If we identify $H_1(N;\Q)$ with $\Q^g$ in the usual way, it is well known (\cite{MP})
that $\rho^+(\Aut\,\pi)$ is of finite index in $\GL(g,\Z)$.

\smallskip
There are related representations of the mapping class group $\M(N)$ on $H^{\pm}$ for which the corresponding statements hold.

\medskip
For later purposes, it is essential to have a very precise description of \linebreak $p_{\ast}: \pi_1(\varSigma,x_0)\rightarrow \pi_1(N,y_0)$.
For simplicity, we restrict from now on to the case that $g$ is even, say $g =2r$. 

\smallskip
We embed $\varSigma$ in $\R^3$ symmetrically as shown in Fig. 1.

\begin{center}
\hspace*{50px} 
\includegraphics[width=264px,height=175px]{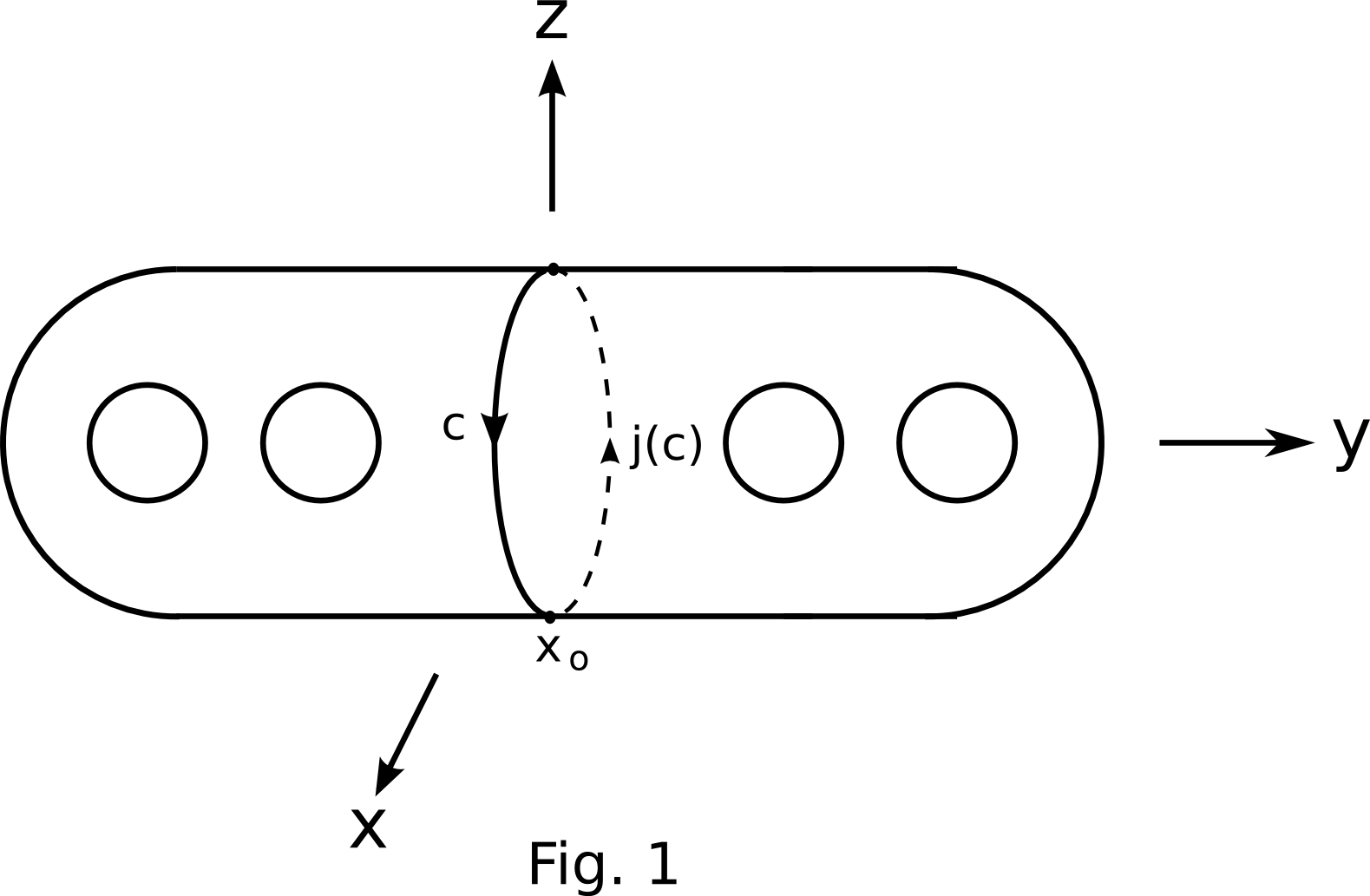}
\end{center}

As deck transformation $j:\varSigma \rightarrow \varSigma$ we take the map given by $j(x,y,z)= (-x,-y,-z)$. On $\varSigma$, we consider the arc
$c$ leading from $j(x_0)$ to $x_0$ and the closed curves $a_1, b_1, \ldots , a_g, b_g$ indicated in Fig. 2.

\begin{center}
 \includegraphics[width=11.5cm,height=3.7cm]{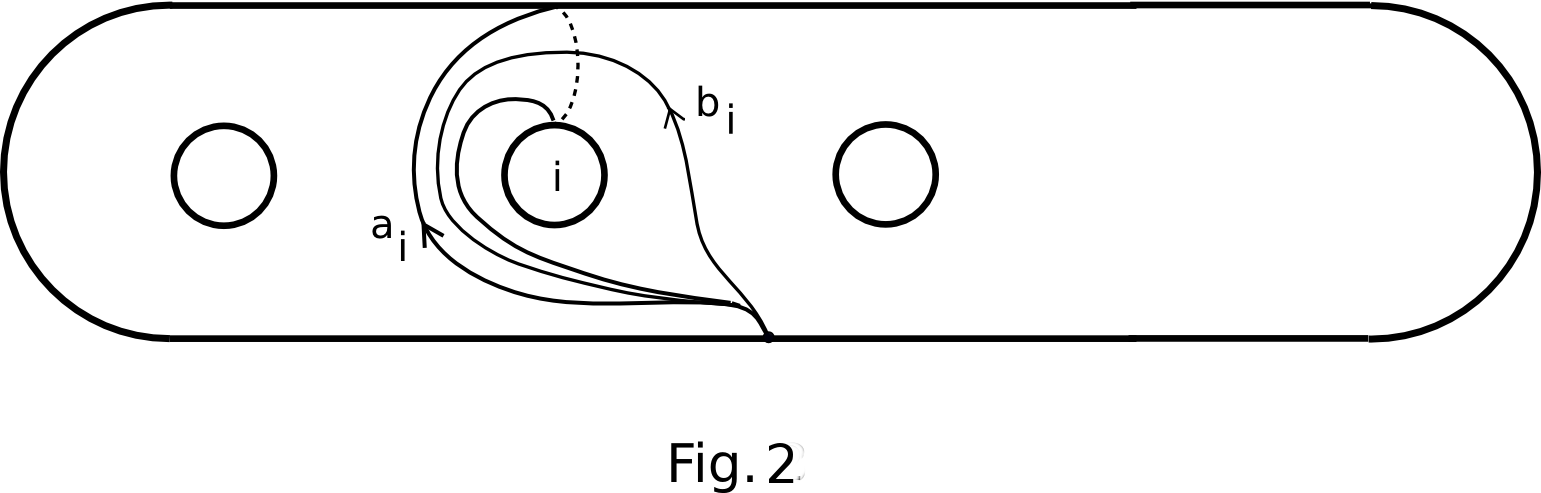}
\end{center}

Then
\[ \pi_1(\varSigma,x_0)=<a_1,b_1,\ldots, a_g,b_g \mid [a_1,b_1] \ldots [a_g,b_g]=1>\]

where $[a,b] := aba^{-1}b^{-1}$ and where we have used the same letter for a curve and its homotopy class. By a further abuse of notation,
we denote the images of $a_1,b_1, \ldots, a_r, b_r,c$ under $p$ and also their homotopy classes in $\pi$ again by the same letters. (Observe
that in $\pi$ we have only half of all the $a_i,b_i$.)

\begin{lemma}\label{pstar}
 We have a presentation 
\[\pi_1(N_{g+1},y_0)=<a_1,b_1,\ldots,a_r,b_r,c \mid [a_1,b_1] \ldots [a_r,b_r]=c^{2}> \;.\]

The homomorphism $p_{\ast}: \pi_1(\varSigma,x_0)\rightarrow \pi_1(N,y_0)$ is given by the following formulas which hold for $1 \leq i \leq r$.

\[\begin{array}{rcl}
 p_{\ast}(a_i)&=&a_i \; ,\\[1ex]
 p_{\ast}(b_i)&=&b_i \; ,\\[1ex]
 p_{\ast}(a_{2r+1-i}) &=& (cb_rb_{r-1}\ldots b_{i+1})\cdot [a_{i+1},b_{i+1}]\ldots [a_r,b_r]\\[1ex]
&& \cdot\, c^{-2}\cdot b_ia_i^{-1}b_i^{-1} \cdot (cb_rb_{r-1}\ldots b_{i+1})^{-1}\; ,\\[1ex]
p_{\ast}(b_{2r+1-i})& =& (cb_rb_{r-1} \ldots b_ia_i)\cdot b_i \cdot (cb_rb_{r-1} \ldots b_ia_i)^{-1}\; .
\end{array}\]
\end{lemma}

\begin{proof}
 Only the last two assertions need proof. They are obtained by elementary geometric considerations.
\end{proof}

\subsection{Products of two cyclic groups} \label{products}

Now we specialize the setting of subsection \ref{general} by taking $G= C_h \times C_k$ for two natural numbers $h,k$. We consider $C_h$
and $C_k$ as subgroups of $\C^\times$ generated by $\zeta:= e^{2\pi i/ h}$ and $\eta:=e^{2 \pi i/k}$. We take $N= N_{2r+1}$ with $r>0$,
denote its fundamental group by $\pi$ and consider the covering space $p:\tilde{N}\rightarrow N$ defined by the homomorphism 
$v:\pi \rightarrow G$ which, with the notation of Lemma \ref{pstar}, is given by $v(a_1)= (\zeta,1)\,,\, v(b_1)= (1, \eta)$; the other generators
are sent to $1$. It is helpful (though perhaps not really necessary) that one can ''see'' the $G$-space $\tilde{N}$ of which we give now
a geometric description:

\smallskip
Consider $G$ as a subgroup of the torus $T= S^1 \times S^1$. Choose a small open disk $D$ around $(1,1)$ in $T$ such that the disks $g \cdot D$
are pairwise disjoint for $g \in G$. Let $T'$ be $T$ with these $hk$ disks removed. On the other hand, consider the surface
\[ X= T\,\sharp \ldots \sharp\, T\,\sharp \,P^2 \]

which is the connected sum of $r-1$ copies of $T$ and of the projective plane. Observe that $N=T\, \sharp\, X$. Let $X'$ be $X$ with one disk removed.
Then $\tilde{N}$ is obtained by equivariantly gluing $X' \times G$ to $T'$ along their boundaries. More precisely, choose a homeomorphism 
$\varphi$ from $\partial X'$ to $\partial D$, extend $G$-equivariantly to obtain a homeomorphism $\varphi$ from $\partial (X' \times G)$ to
$\partial T$ and form
\[\tilde{N}:= T' \cup_{\varphi} (X' \times G) \,.\]

In an obvious manner, we obtain a free operation of $G$ on $\tilde{N}$. 

\medskip
The generators $(\zeta^{-1},1)$ and $(1,\eta^{-1})$ of $G$ define diffeomorphism of $\tilde{N}$ which we denote by $j_1$ and $j_2$. Let us 
simply write $H_1(\tilde{N})$ for the homology group with coefficients in the cyclotomic number field $\Q(\zeta,\eta)$. With $g=2r$, we have
\begin{eqnarray}
 \dim H_1(\tilde{N})= (g-1) \cdot hk +1\,. \label{S5}
\end{eqnarray}

Let $J_i$ be the automorphism $(j_i)_{\ast}$ of $H_1(\tilde{N})$. For $(\alpha, \beta)\in \Z /h \times \Z /k$, put 
\begin{eqnarray}
 H^{\alpha,\beta}:= \{u \in H_1(\tilde{N}) \mid J_1(u)= \zeta^{\alpha} u\,,\, J_2(u)= \eta^{\beta}u\}\;.\label{S6}
\end{eqnarray}

From Proposition \ref{structure} we know that 
for $(\alpha, \beta)\in \Z /h \times \Z /k$, we have 
\begin{equation}\label{S6.1}
\dim H^{\alpha,\beta} = \left \{ \begin{array}{cll}
                                    g& \text{if} & (\alpha,\beta)=(0,0)\;,\\
                                    g-1 & \text{if} & (\alpha,\beta)\neq (0,0)\;.
                                   \end{array}
\right .
\end{equation}

From (\ref{S3}), we get $hk$ homomorphisms 
\begin{eqnarray}
 \rho_v^{H^{\alpha,\beta}}:\Gamma (v)\rightarrow \Aut_{\Q(\zeta,\eta)}\, H^{\alpha,\beta}\;. \label{S12}
\end{eqnarray}

In order to consider $\Aut_{\Q(\zeta,\eta)}\, H^{\alpha,\beta}$ as a matrix group, we have to choose a basis of $H^{\alpha,\beta}$.
We begin by specifying generators for the vector space $H_1(\tilde{N})$.
As usual, let $U$ be the kernel of $v$ and denote by 
\[ q:U \rightarrow U^{\ab} \otimes_{\Z} \Q(\zeta,\eta)= H_1(\tilde{N})\]

the canonical projection. Put 
\begin{eqnarray}
 A:= q(a_1^h)\,, \; B:=q(b_1^k)\;. \label{S7}
\end{eqnarray}

For $d \in \{a_2, b_2, \ldots , a_r, b_r, c\}$ let $D \in \{A_2, B_2, \ldots , A_r, B_r, C\}$ be the corresponding symbol whith capital letter.
Then, for $\nu,\mu \in \Z$, let 
\begin{eqnarray}
 D^{\nu,\mu}:= q(a_1^{\nu}\,b_1^{\mu}\,d\,b_1^{-\mu}\,a_1^{-\nu})\,. \label{S8}
\end{eqnarray}

We have 
\begin{eqnarray}
 J_1(D^{\nu,\mu}) =D^{\nu-1,\mu}\; , \; J_2(D^{\nu,\mu}) =D^{\nu,\mu-1}\; . \label{S9}
\end{eqnarray}

The element $D^{\nu,\mu}$ depends only on the residue classes of $\nu$ mod $h$ and $\mu$ mod $k$.
Therefore, from now on, when we write $D^{\nu,\mu}$, the pair $(\nu,\mu)$ will be an element of $\Z /h \times \Z /k$.

\smallskip
Given our geometric description of $\tilde{N}$, the reader should be able to see all these homology classes. The $(g-1)hk+2$ elements 
$A,B,D^{\nu,\mu}$ generate $H_1(\tilde{N})$; there is one relation, namely 
\begin{eqnarray}
 \sum_{\nu,\mu}\,C^{\nu,\mu}=0 \;. \label{S10}
\end{eqnarray}

From (\ref{S9}), we see:

\begin{prop}\label{basis}
For $(\alpha,\beta)\neq (0,0)$ the elements 
 \begin{eqnarray}
 \hat{D}^{\alpha,\beta}:=\sum_{\nu,\mu}\, \zeta^{\alpha\nu}\,\eta^{\beta\mu}\,D^{\nu,\mu}\;. \label{S11}
\end{eqnarray}
with $D \in \{A_2, B_2, \ldots , A_r, B_r, C\}$ form a basis of $H^{\alpha,\beta}$.\hfill $\Box$
\end{prop}

\medskip

Let us denote by $R$ the ring of integers in $\Q(\zeta,\eta)$. According to Proposition \ref{basis}, we can consider the homomorphism (\ref{S12}) as a
representation 
\begin{eqnarray}
 \rho^{\alpha,\beta}: \Gamma(v) \rightarrow \GL (g-1,R)\;. \label{S13}
\end{eqnarray}

We modify $\rho^{\alpha,\beta}$ to obtain a homomorphism with values in $\SL(g-1,R)$ following \cite{GL}: By the Dirichlet unit theorem,
there exists a torsion free subgroup $F$ of $R^\times$ of finite index. Consider the homomorphism 
\[\begin{array}{lccc}
   \varPhi:& R^\times \times \SL(g-1,R)& \rightarrow & \GL(g-1,R)\\[1ex]
   &(z,A)&\mapsto & zA\;.
  \end{array}\]

Then $\varPhi \mid F \times \SL(g-1,R)$ is injective, and $\Delta:=\varPhi(F\times\SL(g-1,R))$ is of finite index in $\GL(g-1,R)$. Let
\begin{eqnarray}
 \tilde{\Gamma}:=\bigcap_{(\alpha,\beta)\neq(0,0)}(\rho^{\alpha,\beta})^{-1}(\Delta)  \label{S14} 
\end{eqnarray}

and define
\begin{eqnarray}
\tilde{\sigma}^{\alpha,\beta}:= \text{pr}_2 \circ \varPhi^{-1} \circ (\rho^{\alpha,\beta} | \tilde{\Gamma}):\tilde{\Gamma}\rightarrow \SL(g-1,R)\;. \label{S15} 
\end{eqnarray}

Observe that $\tilde{\Gamma}$ is a finite index subgroup of $\Aut \, \pi$. We wish to obtain a representation of a finite index
subgroup of the mapping class group $\Out \, \pi = \Aut \,\pi/ \Inn \,\pi$: Since $G$ is abelian, we see from (\ref{S1}) that 
\[ \Inn\, \pi \subseteq \Gamma(v)\;.\]

\begin{lemma}\label{inner}
 There exists a finite index subgroup $\tilde{\Gamma}'$ of $\tilde{\Gamma}$ such that, for 
  $(\alpha,\beta) \neq (0,0)$, we have 
 \[\tilde{\Gamma}'\cap \Inn \, \pi \, \subseteq \, \ker \,\tilde{\sigma}^{\alpha,\beta}\;.\]
\end{lemma}

\begin{proof}
 Let $\Delta'$ be a torsion-free subgroup of $\SL(g-1, R)$ of finite index (\cite{Se}, Lemma 8) and put 
 \[\tilde{\Gamma}':=\bigcap_{(\alpha,\beta)\neq(0,0)}(\tilde{\sigma}^{\alpha,\beta})^{-1}(\Delta')\,.\]

 For $\delta \in \Inn \,\pi$, it is clear that $\rho^{\alpha,\beta}(\delta)=\gamma\, I_{g-1}$ with $\gamma^{hk}=1$. If in addition
 $\delta \in \tilde{\Gamma}$, we have also $\rho^{\alpha,\beta}(\delta)=zA$ with $z \in F$ and $A \in \SL(g-1,R)$, and by looking at
 determinants, we see that $z$ is torsion, hence $z=1$ and $\rho^{\alpha,\beta}(\delta) \in \SL(g-1,R)$. Therefore
 $\tilde{\sigma}^{\alpha,\beta}(\delta)= \rho^{\alpha,\beta}(\delta)$ is torsion. The assertion follows.
\end{proof}

Denoting the image of $\tilde{\Gamma}'$ in $\Out\, \pi$ by $\Gamma$, the homomorphism $\tilde{\sigma}^{\alpha,\beta}$ induces
by Lemma \ref{inner} for $(\alpha,\beta) \neq (0,0)$  a representation 
\begin{eqnarray}
 \sigma^{\alpha,\beta}: \Gamma \rightarrow \SL(g-1,R)\;. \label{S16}
\end{eqnarray}

\section{Explicit formulas for automorphisms of the fundamental group} \label{explicit}

As in subsection \ref{orient}, we consider the orientation covering $p:\varSigma \rightarrow N$ of $N= N_{g+1}$. We fix points $y_0 \in N$
and $x_0 \in p^{-1}(y_0)$ and continue to write $\pi:= \pi_1(N,y_0)$. There is a canonical isomorphism between the homeomorphism group $\Ho\,N$
and the subgroup 
\begin{eqnarray}
 S⁺ :=\{h \in \Ho\,\varSigma \mid h \; \text{is orientation preserving and}\; hj =jh\} \label{S17}
\end{eqnarray}

of $\Ho \, \varSigma$. This isomorphism induces, for $g\geq 2$, an isomorphism between the mapping class group $\M(N)$ and a subgroup of
$\M (\varSigma)$. Let 
\begin{eqnarray}
 S_0^+:=\{h \in S⁺ \mid h(x_0) = x_0\}\;. \label{S18}
\end{eqnarray}

Let us recall explicitly a finite set of elements of $S_0^+$ which, according to \cite{BC}, are generators for $\M\,N$:

For a submanifold $\gamma$ of $\varSigma$ diffeomorphic to $S^{1}$, we denote by $T_{\gamma} \in \M (\varSigma)$
the Dehn twist along $\gamma$; we stick to the orientation convention of \cite{Li} and \cite{BC}.
Restricting to the case $g=2r>0$, we have to consider the Dehn twists along the curves $\alpha_i$ $(1 \leq i \leq g)$,
$\beta_i$ $(1 \leq i \leq g)$, $\gamma_i$ $(1 \leq i \leq g-1)$ and $\varepsilon$ shown in Fig.3. (The notation of \cite{BC} differs from
that of \cite{Li}; ours differs from both.) 

\begin{center}
 \includegraphics[width=11.5cm,height=4.5cm]{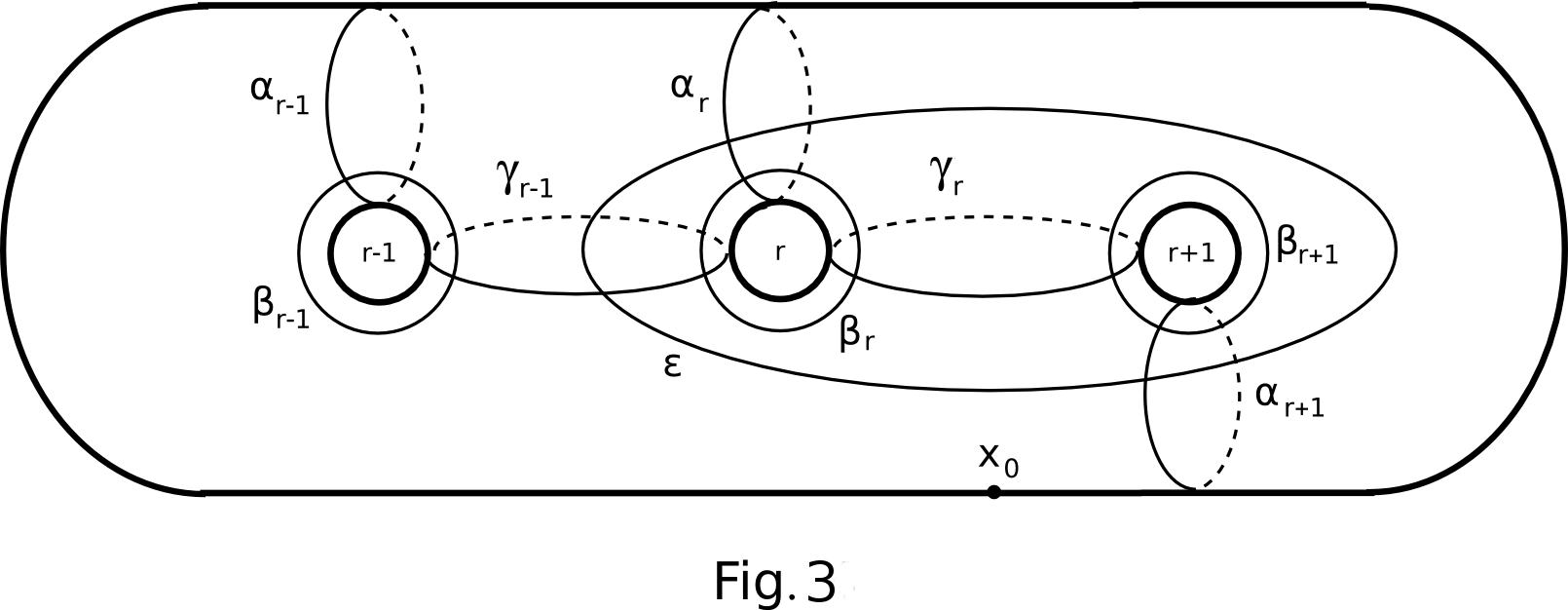}
\end{center}

We consider the finite subset of $S^+_0$ which consists of the elements
\begin{itemize}
 \item $\tilde{R}_i:=T_{\textstyle\alpha_i}\; {\scriptstyle \circ}\; T^{-1}_{\textstyle \alpha_{g+1-i}}\; ,\quad 1 \leq i \leq r\; ,$
 \item $\tilde{S}_i:=T_{\textstyle\beta_i}\; {\scriptstyle \circ}\; T^{-1}_{\textstyle \beta_{g+1-i}}\; ,\quad 1 \leq i \leq r\; ,$
 \item $\tilde{T}_i:=T_{\textstyle\gamma_i}\; {\scriptstyle \circ}\; T^{-1}_{\textstyle \gamma_{g-i}}\; ,\quad 1 \leq i \leq r-1\; ,$
 \item one further homeomorphism $\tilde{Y}$, given by the following lemma.
\end{itemize}

\begin{lemma}\label{ypsilon}
 There is, up to isotopy, a unique homeomorphism $\tilde{Y}$ of $\varSigma$ fixing $x_0$ and commuting with $j$ which is isotopic
to
\[(T_{\gamma_{r}}T_{\beta_{r}}T_{\beta_{r+1}})^{2}\,T^{-1}_{\varepsilon}\]
by an isotopy fixing $x_0$ and $j(x_0)$.
\end{lemma}

\begin{proof}
 The proof of this lemma, without its last line, is sketched on pages 440-441 of \cite{BC}. The diligent reader can fill in the 
details and keep track of the support of the isotopies. 
\end{proof}

These elements define homeomorphisms of $N$ and these in turn automorphisms of
$\pi$  which we denote respectively by $R_i$ $(1\leq i \leq r)$, $S_i$ $(1\leq i \leq r)$,
$T_i$ $(1\leq i \leq r-1)$ and $Y$. These together with the inner automorphisms generate the group $\Aut\, \pi$. Using Lemma \ref{pstar}, 
it is easy to see how they operate on the generators $a_1,b_1, \ldots , a_r,b_r,c$ of the group $\pi$ (for brevity, we write only
the values of those generators which are not kept fixed):

\begin{prop}\label{autom}
In the group $\pi = \pi_1(N_{2r+1},y_0)$ we introduce the elements
\[w_i:= b_i a_i^{-1} b_i^{-1} a_{i+1}\;, \; 1 \leq i \leq r-1 \; .\]
Furthermore, we write $a:=a_r$ and $b:=b_r$. Whith these notations, we have 

\begin{itemize}
\item  $R_i(b_i) = b_i a_i$

\medskip
\item $S_i(a_i)=a_i b_i^{-1}$
 
\medskip
\item $T_i(b_i)=w_i^{-1}b_i$
\item[] $T_i(a_{i+1})=w_i^{-1}a_{i+1}w_i $
\item[] $T_i(b_{i+1})=b_{i+1}w_i$

\medskip
\item $Y(a)=b^{-1} c^{-1} b a^{-1} b^{-1} c^{-1} b$
\item[] $Y(b)=b^{-1} c b a b a^{-1} b^{-1} c^{-1} b$
\item[] $Y(c)=c^{2} b a b^{-1} a^{-1} b^{-1} c^{-1} b$ \hfill $\Box$
\end{itemize}

\end{prop}

All the results of section \ref{explicit} have analogs in the case $g=2r+1$.

\section{Proof of the main theorem}\label{one}
We specialize the situation of section \ref{products} by taking $h=k\geq 2$. The representation 
$\sigma$ of the theorem  will be obtained from 
\begin{eqnarray}
 \rho:=\prod_{\alpha \in \Z/k} \rho^{\alpha,1}:\Gamma(v) \rightarrow \GL(g-1, R)^k \label{S19}
\end{eqnarray}

where $\rho^{\alpha,1}$ was introduced in (\ref{S13}). Our task will be to find sufficiently many elements of $\Gamma(v)$ which have the property
introduced in the following definition. 

\medskip
\textbf{Definition.} Let $\alpha \in \Z/k$. An element $\gamma \in \Gamma(v)$ is called $\alpha$-\emph{concentrated} if $\rho^{\delta,1}(\gamma)=1$
for $\delta \neq \alpha$.

\smallskip
Let $\Gamma^{\alpha}$ be the subgroup of $\Gamma(v)$ which consists of all $\gamma$ which are $\alpha$-concentrated and satisfy
$\det(\rho^{\delta,\beta}(\gamma))=1$ for all $(\delta,\beta )\neq (0,0)$.

\medskip
With these notations, we have the following lemma.

\begin{lemma}\label{elementary}
 \emph{(a)} $\rho^{\alpha,1}(\Gamma^{\alpha})= \rho^{0,1}(\Gamma^0)$ for $\alpha \in \Z/k$.
 
 \smallskip
 \emph{(b)} $\rho^{0,1}(\Gamma(v))$ is a $\Gal\,(\Q(\zeta)/\Q)$-invariant subgroup of $\GL(g-1,R)$ where the Galois group acts on all entries of a matrix.
\end{lemma}

\begin{proof}
 We consider the automorphisms $S_1$ and $R_1$ of $\pi$  introduced in section \ref{explicit}. They have the following two properties:
\begin{eqnarray}
 S_1(U)=U =R_1(U)\;.\quad  \label{S20}\\[1ex]
 S_1 \;\text{and}\; R_1 \;\text{normalize}\; \Gamma(v)\;. \label{S21}
\end{eqnarray}

Property (\ref{S20}) follows immediately from Proposition \ref{autom}. For (\ref{S21}), let $v_1,v_2: \pi \rightarrow C_k$ be the two
components of $v$ and observe that 
\[v \circ S_1= (v_1,v_2 \cdot v_1^{-1})\; , \; v\circ R_1 = (v_1 \cdot v_2,v_2)\; ;\]

then (\ref{S21}) follows easily from (\ref{S1}).

\smallskip
Because of (\ref{S20}), $S_1$ and $R_1$ induce  automorphisms $\bar{S}_1$ and $\bar{R}_1$ of $H_1(\tilde{N})$ $= U^{\ab}\otimes \Q(\zeta)$. With the notation from
(\ref{S8}) and (\ref{S11}), we have 
\begin{eqnarray}
\bar{S}_1(D^{\nu,\mu})= D^{\nu,\mu - \nu}\;,\nonumber \\[1ex]
 \bar{S}_1(\hat{D}^{\alpha,\beta})= \hat{D}^{\alpha+\beta,\beta}\, .\label{S22}
\end{eqnarray}

From (\ref{S21}), we see that, for $\gamma \in \Gamma(v)$, we can form $\rho^{\alpha,\beta}(S_1^{-1}\gamma S_1)$; formula (\ref{S22})
then gives the matrix equation 
\begin{eqnarray}
 \rho^{\alpha,\beta}(S_1^{-1}\gamma S_1) = \rho^{\alpha+\beta,\beta}(\gamma) \label{S23}
\end{eqnarray}

which holds for $(\alpha,\beta)\neq (0,0)$. This shows (a). 

\medskip
Proceeding in an analogous manner for $R_1$, we have 
\begin{eqnarray}
 \rho^{\alpha,\beta}(R_1^{-1}\gamma R_1) = \rho^{\alpha,\beta -\alpha}(\gamma) \label{S24}
\end{eqnarray}
Let $\psi \in \Gal\,(\Q(\zeta)/\Q)$ with $\psi(\zeta)= \zeta^\ell$. Then
\begin{eqnarray}
 \psi(\rho^{\alpha,\beta}(\gamma)) = \rho^{\ell\alpha,\ell\beta}(\gamma)\,.\label{S25}
\end{eqnarray}
By (\ref{S23}) and (\ref{S24}), we can find an element $w \in <R_1,S_1>$ such that $\rho^{0,\ell}(w^{-1}\gamma w)$ $=\rho^{0,1}(\gamma)$,
and by (\ref{S25}), we have $\psi(\rho^{0,1}(w^{-1}\gamma w)) = \rho^{0,1}(\gamma)$. This shows (b).
\end{proof}

\medskip
Our main theorem as stated in the introduction will follow easily from 
\begin{lemma}\label{key}
 The index of $\rho^{0,1}(\Gamma^0)$ in $\SL(g-1,R)$ is finite.
\end{lemma}

Let us show how Lemma \ref{key} implies the theorem: The representation $\rho$ of (\ref{S19}) defines, by the process described at the end of
section \ref{products}, a homomorphism
\begin{eqnarray}
\sigma:=\prod_{\alpha \in \Z/k} \sigma^{\alpha,1}:\Gamma \rightarrow \SL(g-1, R)^k\;. \label{S26}
\end{eqnarray}

We show that $\sigma$ has an image of finite index. 
Let $\Gamma'$ be the subgroup of $\Gamma(v)$ generated by all the $\Gamma^{\alpha}$ for $\alpha \in \Z/k$. Then, by (\ref{S19}), by Lemma \ref{elementary} (a) 
and Lemma \ref{key}, we see that $\rho(\Gamma')$ is of finite index in $\SL(g-1, R)^k$. We don't assert that $\Gamma'$ is of finite index in
$\Aut\, \pi$. But $\Gamma'$ is obviously contained in the group $\tilde{\Gamma}$ of (\ref{S14}). Hence the group $\sigma(\Gamma)$ is of
finite index. \hfill $\Box$

\medskip
In order to prove Lemma \ref{key}, we shall show that $\rho^{0,1}(\Gamma^0)$ contains sufficiently many elementary matrices. We adopt the usual
notation: For $i\neq j$ and $z \in R$, let $E_{i,j}(z)$ be the matrix which differs from the unit matrix only in the position $(i,j)$ where we 
have $z$ instead of $0$. Instead of $E_{i,j}(1)$, we write simply $E_{i,j}$. The elementary matrices satisfy some well-known commutator relations,
the so-called Steinberg relations (see \cite{GL}, (42)), which we will use quite often. To simplify some formulas, it is convenient to replace
the representation $\rho^{0,1}$ by 
\begin{eqnarray}
 \varrho: \Gamma (v)\rightarrow \GL(g-1,R)\, , \quad \gamma \mapsto E^{-1}_{g-1,r-1} \cdot \rho^{0,1}(\gamma) \cdot E_{g-1,r-1}\; . \label{S27}
\end{eqnarray}

It is enough to show that the index of $\varrho(\Gamma^0)$ in $\SL(g-1,R)$ is finite.

\smallskip
Recall the elements of $\Aut\, \pi$ described in Proposition \ref{autom}. All of them with the exception of $R_1,S_1,$ and $T_1$ lie in the 
subgroup $\Gamma(v)$. There are two further interesting automorphisms of $\pi$:
\begin{eqnarray}
 V&:= &T_1^k \; ,\label{S28} \\[1ex]
 W&:=&(S_1 T_1 \ldots S_{r-1} T_{r-1})\cdot Y \cdot (S_1 T_1 \ldots S_{r-1} T_{r-1})^{-1} \, . \label{S29}
\end{eqnarray}

By straightforward but onerous computations, one can show: 
\begin{lemma}\label{v}
 We have $V \in \Gamma^0$ and 
 \[\varrho(V)= E_{1,r}(k) \, E_{g-1,r}(z)\]
 with $z = 2k(1-\zeta)^{-1} = -2 (1+2\zeta + \ldots + k\zeta^{k-1})$. \hfill $\Box$
\end{lemma}

\begin{lemma}\label{w}
We have $W \in \Gamma(v)$ and 

\parbox{11.8cm}{
\begin{eqnarray*}\varrho(W)=\left[\arraycolsep2.8pt\begin{array}{cccc|cccc|c}
1\,&      &&&&&&&\\
 &\ddots\, &&& & &&&\\
 & &\;1 &&&& &&\\
 -2\,&\cdots &-2\,&-\zeta&&&&1+\zeta&1-\zeta\\ \hline 
 \rule[-2mm]{0mm}{6mm}&&&&1&&&&\\
 &&&&&\quad\ddots&&&\\
 &&&&&&\quad\ddots &&\\
 &&&&&&&1&\\ \hline
 &&&&&& &&1
\end{array} \right]\;.\end{eqnarray*}}\hfill
\parbox{6mm}{$\Box$} 
\end{lemma}

In contrast to the two previous lemmas, the next lemma is completely obvious. It allows to produce further elements of $\Gamma^0$:

\begin{lemma}\label{comm}
 For $\gamma \in \Gamma(v)$ and $\gamma_0 \in \Gamma^0$, we have $[\gamma,\gamma_0 ] \in \Gamma^0$. \hfill $\Box$
\end{lemma}

We will deduce Lemma \ref{key} from the following lemma.

\begin{lemma}\label{steinberg}
 For all $i,j \in \{1, \ldots , g-1\}$ with $i\neq j$ there exist $\gamma_{ij} \in \Gamma(v)$ and an integer $m_{ij}\neq 0$ such that
 $\varrho(\gamma_{ij})= E_{ij}(m_{ij})$.
\end{lemma}

\begin{proof} 
Let $s:= r-1$. For $1 \leq i,j \leq 2s+1$ with $i \neq j$ we have to prove the following assertion:
 \[\mathfrak{A}_{i,j}: \exists \, \gamma \in \Gamma(v), \, m \in \Z \smallsetminus \{0\}:  \varrho(\gamma)= E_{ij}(m) \,.\]
  
We first compute $\varrho$ on several basic elements:

\begin{align}
 \varrho (R_{i+1})&= E_{i,i+s}\, ,\quad 1\leq i\leq s-1\,, \label{S30}\\
 \varrho (R_{s+1})&= E_{s,2s}\, E_{2s+1,2s}^{-1}\, ,        \label{S31}\\
 \varrho (S_{i+1})&= E_{i+s,i}^{-1}\, , \quad 1\leq i \leq s\, , \label{S32}\\
 \varrho (T_{i+1})&= E_{i,i+s}\,E_{i+1,i+1+s}\,E_{i,i+1+s}^{-1}\, E_{i+1,i+s}^{-1}\,,\quad 1\leq i \leq s-2 \,, \label{S33}\\
 \varrho (T_{s})&= E_{s-1,2s-1}\, E_{s,2s}\,E_{s-1,2s}^{-1}\,E_{s,2s-1}^{-1}\,E_{2s+1,2s-1}\,E_{2s+1,2s}^{-1} \;\text{if}\; r\geq 3, \label{S34}\\
 \varrho(Y)&= D_s \,, \label{S35}
\end{align}

 where $D_i$ denotes the diagonal matrix diag$(1,\ldots , 1,-1,1, \ldots ,1)$ with $-1$ sitting in the $i$-th place. Next, we will prove the
 following rules which explain how the assertions $\mathfrak{A}_{i,j}$ are related to each other.
 
 \begin{itemize}
  \item [(i)] $\mathfrak{A}_{i,l} $ and $ \mathfrak{A}_{l,j} \Rightarrow \mathfrak{A}_{i,j} $ for $ i,j,l $ pairwise different.
  \item [(ii)] $\mathfrak{A}_{i,j+s} \Rightarrow \mathfrak{A}_{i,j}$ for $1 \leq i,j \leq s$ and $i \neq j$.
  \item [(iii)] $\mathfrak{A}_{i,s+1} \Rightarrow \mathfrak{A}_{i+s,s+1}$ for  $2\leq i \leq s\,.$
  \item [(iv)] $\mathfrak{A}_{s,i+s} \Rightarrow \mathfrak{A}_{s,i+1+s}$ for $1 \leq i \leq s-1\,.$
  \item [(v)] $\mathfrak{A}_{i+1,s+1} \Rightarrow \mathfrak{A}_{i,s+1}$ for $2 \leq i \leq s-1\,.$
 \end{itemize}

 Rule (i) is an easy consequence of the Steinberg relations. By (\ref{S32}) the assertion $\mathfrak{A}_{s+j,j}$ is true for $1 \leq j \leq s$,
 so rules (ii) and (iii) follow from (i). To prove (iv) and (v), consider the elements $L_i:= \varrho(T_{i+1}^{-1}\,R_{i+1}\,R_{i+2})\,,
 1\leq i\leq s-1$. We have 
 
 \begin{equation*}
  L_i =
  \begin{cases}
   E_{i,i+1+s}\,E_{i+1,i+s}\,,& 1\leq i \leq s-2\,,\\
   E_{s-1,2s}\, E_{s,2s-1}\,E_{2s+1,2s-1}^{-1} \,, & i=s-1\,.
  \end{cases}
 \end{equation*}

 Then (iv) follows from rule (ii) and the Steinberg relation
 \[ [\,E_{s,i}(m),L_i\,] = E_{s,i+1+s}(m)\,.\]
 
 Analogously, (v) can be deduced from rule (iii) and the relation 
 \[ [\,L_i, E_{i+1+s,s+1}(m)\,] = E_{i,s+1}(m) \,.\]
 
 Now we want to see how the assertions $\mathfrak{A}_{i,j}$ can be proven with the help of the rules (i)-(v). We compute 
 $\varrho([\,W,R_2^{-1}])= E_{s,s+1}(2)$, so $\mathfrak{A}_{s,s+1}$ is true. We deduce from (iv) and (v) that $\mathfrak{A}_{s,s+i}$
 is true for $1 \leq i \leq s$ and $\mathfrak{A}_{i,s+1}$ is true for $2 \leq i \leq s$. Note that this together with (\ref{S30}) and
 (\ref{S32}) implies that $\mathfrak{A}_{i,i+s}$ and $\mathfrak{A}_{i+s,i}$ are true for all $1 \leq i \leq s$ which give us by the 
 Steinberg relations the rule that for any distinct $i,j$ with $1 \leq i,j \leq s$, the assertions
 \[\mathfrak{A}_{i,j}, \mathfrak{A}_{i,j+s}, \mathfrak{A}_{i+s,j}, \mathfrak{A}_{i+s,j+s} \;\text{are equivalent.}\]
 
 This gives us together with (i) that for all distinct $i,j$ with $1 \leq i,j \leq 2s$ the assertions $\mathfrak{A}_{i,j}$ are true.\\
 Note that (\ref{S31}) and $\mathfrak{A}_{s,2s}$ together imply that $\mathfrak{A}_{2s+1,2s}$ is true and therefore $\mathfrak{A}_{2s+1,j}$
 is true for all $j \leq 2s$ by (i) and $\mathfrak{A}_{2s,j}$.
 
 \smallskip
 By (\ref{S35}) and Lemma \ref{w}, it is easy to compute $\varrho([\,W,Y\,])$. The product of the Galois conjugates of $\varrho([\,W,Y\,])$
 is of the form
 
 \begin{eqnarray*}
 \left[\arraycolsep2.8pt\begin{array}{cccc|cccc|c}
 1\,&      &&&&&&&\\
 &\ddots\, &&& & &&&\\
 & &\;1 &&&& &&\\
 \lambda\,&\cdots &\lambda\,&1&&&& \mu & \nu\\ \hline 
 \rule[-2mm]{0mm}{6mm}&&&&1&&&&\\
 &&&&&\quad\ddots&&&\\[0.5ex]
 &&&&&&&1&\\ \hline
 &&&&&& &&1
\end{array} \right]\end{eqnarray*}

with rational integers $\lambda,\mu, \nu$ and $\nu \neq 0$. By Lemma \ref{elementary} this matrix is also in the image of $\varrho$.
Since $\mathfrak{A}_{s,1},\ldots ,\mathfrak{A}_{s,s-1}$ and $\mathfrak{A}_{s,2s}$ are true, this implies that the assertion $\mathfrak{A}_{s,2s+1}$
is true. From (i) we deduce that $\mathfrak{A}_{i,2s+1}$ is true for all $i \leq 2s$.
\end{proof}

\emph{Proof of Lemma \ref{key}.} First observe that by (\ref{S30}), (\ref{S31}) and Lemma \ref{w} we have 
\[\varrho ([\,R_2,W^{-1}\,]) = E_{s,s+1}(2 \bar{\zeta})\;.\]

By Lemma \ref{elementary}, we conclude that $E_{s,s+1}(2\zeta ) =\varrho (X)$ with $X \in \Gamma(v)$. Consider the elements $\gamma_{ij} \in \Gamma(v)$
from Lemma \ref{steinberg}. Forming iterated commutators of suitable $\gamma_{ij}$'s with $X$, we see that there is an integer $m\neq 0$ and
elements $\delta_{ij} \in \Gamma(v)$ with $\varrho(\delta_{ij})= E_{ij}(m\zeta)$. For $j\neq r$, we have $\varrho([\,\delta_{j,1},V\,])=E_{j,r}(km\zeta)$
by Lemma \ref{v}, and $[\,\delta_{j,1},V\,] \in \Gamma^0$ by Lemma \ref{comm}. Again from Lemma \ref{comm} and the Steinberg relations, we deduce
that there is an integer $n\neq 0$ such that all the matrices $E_{i,j}(n\zeta^{\ell})$ are contained in $\varrho(\Gamma^0)$. By \cite{Va}, the index
of $\varrho(\Gamma^0)$ in $\SL(g-1,R)$ is finite. \hfill $\Box$

\bigskip
Mathematisches Institut\\ der Heinrich-Heine-Universit\"at Düsseldorf \\
Universit\"atsstra\ss e 1\\
D-40225 Düsseldorf

\medskip
deniz@math.uni-duesseldorf.de\\
singhof@math.uni-duesseldorf.de

\end{document}